\documentclass[11pt,twoside,openright]{amsart}        

\usepackage{amsfonts}
\usepackage{amsmath}
\usepackage{amsthm}
\usepackage{amscd}
\usepackage{euscript}
\usepackage{amssymb}
\usepackage{psfrag}
\usepackage{graphicx}   
\usepackage{rotating}
\usepackage{color}
\usepackage{mathrsfs}
\usepackage{xypic}
\xyoption{all}
\theoremstyle{plain}
\newtheorem{thm}{Theorem}[section]

\newtheorem{lem}[thm]{Lemma}
\newtheorem{prop}[thm]{Proposition}
\theoremstyle{definition}
\newtheorem{defn}{Definition}[section]
\theoremstyle{remark}
\newtheorem{rem}{Remark}[section]
\theoremstyle{example}
\newtheorem{ex}{Example}[section]
\numberwithin{equation}{section}

\newcommand{\Z}{\mathbb Z}
\newcommand{\F}{\mathbb F}
\newcommand{\C}{\mathbb C}

\newcommand{\PP}{{\mathbb P}}

\newcommand{\OO}{\mathcal{O}}
\newcommand{\FF}{\mathcal{F}}

\newcommand{\EE}{\mathcal{E}}

\newcommand{\Q}{\mathcal{Q}}

\DeclareMathOperator{\HH}{H} 
\DeclareMathOperator{\Ext}{Ext} \DeclareMathOperator{\Hom}{Hom}
 \DeclareMathOperator{\rk}{rk}
\DeclareMathOperator{\Ker}{Ker} 
\DeclareMathOperator{\coker}{Coker} 
 \DeclareMathOperator{\Id}{Id}

 \DeclareMathOperator{\Lie}{Lie}

 \DeclareMathOperator{\SL}{SL}

 \DeclareMathOperator{\gr}{gr}

\DeclareMathOperator{\E}{\underbar{E}}

\DeclareMathOperator{\lieg}{\mathfrak{g}}
\DeclareMathOperator{\liep}{\mathfrak{p}}

\DeclareMathOperator{\lieh}{\mathfrak{h}}
\DeclareMathOperator{\lier}{\mathfrak{r}}
\DeclareMathOperator{\lien}{\mathfrak{n}}

\title{Sections of homogeneous vector bundles}
\subjclass[2000]{14F05,14M17, 16G20}
\keywords{Rational Homogeneous Variety, Homogeneous Vector Bundle, Cohomology, Quiver Representation}

\author{Ada Boralevi}
\address{%
Department of Mathematics\\
Mailstop 3368\\
Texas A$\&$M University\\
College Station, TX 77843-3368\\
	USA}
\email{boralevi@math.tamu.edu}

\begin{document}

\maketitle

\begin{abstract}
In this work we give a method for computing sections of homogeneous vector bundles on any rational homogeneous variety $G/P$ of type $ADE$. Our main tool 
is the equivalence of categories between homogeneous vector bundles on $G/P$ and finite dimensional representations of a given quiver with relations. 
Our result generalizes the work of Ottaviani and Rubei \cite{OR}. 
\end{abstract}

\section{Introduction}
The Borel-Weil-Bott theorem computes the cohomology of irreducible homogeneous vector bundles on rational homogeneous varieties. In this paper we give a method for computing sections of homogeneous vector bundles in the non-irreducible case. Our result holds on rational homogeneous varieties of type $ADE$ and generalizes the work of Ottaviani and Rubei \cite{OR} holding on Hermitian symmetric varieties of type $ADE$. 

The category of homogeneous vector bundles on a rational homogeneous variety $G/P$ is equivalent to the category of $P$-modules and also to that of integral $\liep$-modules. 
The Borel-Weil-Bott theorem exploits this equivalence: the cohomology of an irreducible homogeneous bundle $E_\lambda$ is determined by the Weyl orbit of the maximal weight of its associated representation. Our main tool is instead the equivalence of categories relating homogeneous vector bundles on $G/P$ with finite dimensional representations of a given quiver with relation $\Q_{G/P}$.

The original idea of using quivers to study homogeneous bundles is due to Bondal and Kapranov \cite{BK} and was later refined by Hille \cite{Hi1}. In \cite{ACGP} \'{A}lvarez-C\'{o}nsul and Garc\'{i}a-Prada gave an equivalent construction, while in \cite{OR} Ottaviani and Rubei used the quiver for computing cohomology.

The cohomology of a homogeneous bundle $E$ on a Hermitian symmetric variety of type $ADE$ is obtained in \cite{OR} as cohomology of a complex $(\HH^*(\gr E), c_*)$. Here $\gr E$ is the associated graded bundle, obtained by restricting the $P$-module associated with $E$ to its Levi factor $R$. The maps $c_i$ are constructed by composing the quiver representation maps of $[E]$ along the segments connecting any vertex with its mirror image in the adjacent Bott chambers. In this paper we generalize this construction to all $ADE$ rational homogeneous varieties, but we limit ourselves to the first cohomology group $\HH^0(E)$. In the last part of the paper we give explicit examples of why Ottaviani and Rubei's result could not be fully generalized. We conjecture that in fact their result does not hold in the general non-Hermitian symmetric case.

The key point in our construction is the use of a special type of homogeneous bundles, that we call $A_m$-type bundles. Namely these are bundles whose quiver representation has support on an $A_m$-type quiver, and thus they have nice splitting properties entailed by quiver theory. In some sense, they play in our proof the same role as the $\PP^1$ fibration does in 
Demazure's proof \cite{Dem} of the Borel-Weil-Bott theorem.

We start by giving a method for computing sections of $A_m$-type bundles. The key point is that we can interpret the coboundary map in cohomology as a map in quiver representations.\\

\textbf{Theorem A.} \emph{Let $S$ be an $A_m$-type bundle on a rational homogeneous variety $G/P$ of type $ADE$, and suppose that $\HH^0(\gr S) \neq 0$. Let $E_\lambda$ be an irreducible summand of $\gr S$ such that $\HH^0(E_\lambda)=V$, for $V=\Sigma^\lambda$ nonzero irreducible $G$-module. Then:
\begin{enumerate}
	\item If among the summands of $\gr S$ there is an $E_\mu$ such that $\HH^1(E_\mu)=V$, consider in the quiver $\Q|_S$ the path from $E_\lambda$ to $E_\mu$. By composing the linear maps corresponding to this path in the representation $[S]$, we get a linear map $\sigma_0^V: V_\lambda \rightarrow V_\mu$. Then the isotypical component $\HH^0(S)^V=V \otimes (\Ker \sigma_0^V)$. 
	\item If there is no such $E_\mu$, then $\HH^0(S)^V= \HH^0(E_\lambda \otimes V_\lambda)=V^{\oplus \dim V_\lambda} $. 
\end{enumerate}}

Theorem A already shows some advantage of our method over the usual method of spectral sequences. In order to compute sections of an $A_m$-type bundle $S$ it is enough to compute the maps of the associated quiver representation once and for all.  

We then use the result of Theorem A to give a construction that allow us to extend the result to any homogeneous bundle on $G/P$. Namely, given any homogeneous bundle $E$ on $X$, every time we have a pair $E_\lambda$, $E_\mu$ of irreducible summands of $\gr E$ such that $\HH^0(E_\lambda)=\HH^1(E_\mu)$, we can construct a distinguished isomorphism $j_{\lambda\mu}:\HH^0(E_\lambda) \rightarrow \HH^1(E_\mu)$ (the details are contained in Lemma \ref{distinguished iso} ).

Our main result is the following:\\

\textbf{Theorem C.} \emph{Let $E$ be a homogeneous vector bundle on a rational homogeneous variety $G/P$ of type $ADE$, and 
consider pairs $E_\lambda$, $E_\mu$ of irreducible summands of the graded $\gr E$ such that $\HH^0(E_\lambda)=\HH^1(E_\mu)$.\\ Define the maps
$c_{\lambda\mu}:\HH^0(E_\lambda \otimes V_\lambda) \rightarrow \HH^1(E_\mu \otimes V_\mu)$
to be the tensor product of the distinguished isomorphism $j_{\lambda\mu}$ with the composition of the maps $V_\lambda \rightarrow V_\mu$ in the quiver representation.
Putting together all these maps for all the possible pairings $(\lambda,\mu)$ we get a map $c_0:=\sum_{\lambda,\mu}c_{\lambda\mu}$, $c_0:\HH^0(\gr E) \rightarrow \HH^1(\gr E)$. Then $\HH^0(E)=\Ker c_0$}\\

One of the advantages of our construction is that one only has to deal with maps $\HH^0(\gr E) \rightarrow \HH^1(\gr E)$, whereas when using spectral sequences it is often necessary to compute $\HH^2(\gr E)$ and the related maps as well.\\

The paper is organized as follows. We set our notation and recall the Borel-Weil-Bott Theorem in Section 2.1. We describe the quiver $\Q_X$ with relations, its representations and the equivalence of categories with homogeneous bundles in Sections 2.2 and 2.3. Then we move on to cohomology computations: in Section 3.2 we treat the case of $A_m$-type bundles, and in Section 3.3 we state and prove our main result giving a method to compute sections of homogeneous bundles. In the last Section 4 we describe in detail an example on the flag manifold $\SL_3/B$ illustrating why our construction cannot be generalized to higher cohomology.

\section{Preliminaries}
\subsection{First equivalence of categories and Bott theorem}\label{first}
Let $G$ be a complex semisimple Lie group of $ADE$ type. 
We make a choice of simple roots $\Delta=\{\alpha_1,\ldots,\alpha_n\}$ of $\lieg=\Lie G$, and we call $\Phi^+$ (respectively $\Phi^-$) the set of
positive (respectively negative) roots. We denote by $\lieh \subset
\lieg$ the Cartan subalgebra. 
The Killing product $(\:\:\:,\:)$ allows us to identify $\lieh$ with $\lieh^*$, and thus define the 
Killing product on $\lieh^*$ as well. Let $\{\lambda_1,\ldots,\lambda_n\}$ be the fundamental weights corresponding to $\{\alpha_1,\ldots,\alpha_n\}$, that is, the elements of $\lieh^*$ such that 
$2(\lambda_i,\alpha_j)/(\alpha_j,\alpha_j)=\delta_{ij}$. Let $Z$ be the lattice generated on $\Z_{\geq 0}$ by the fundamental weights. The elements in $Z$ are called the dominant weights of $G$, and they are maximal weights 
of the irreducible representations of $\lieg$. In the $ADE$ case, all roots have length $\sqrt{2}$.

For any $V$ representation of $G$, we denote by $V^G$ its invariant part, that is, the subspace of $V$ where $G$ acts trivially. If $\Sigma$ 
is an irreducible representation, we denote $V^\Sigma:=\Hom(\Sigma,V)^G \otimes \Sigma$.

Finally, let $X=G/P$ be a rational homogeneous variety, where $P \leq G$ is a parabolic subgroup. Our aim is studying the cohomology of homogeneous vector bundles on $X$.

A vector bundle $E$ on $G/P$ is called homogeneous if there is an action of $G$ on $E$ such that the following diagram commutes:
$$\xymatrix{G \times E \ar[r]\ar[d]&E\ar[d]\\
G \times G/P \ar[r]& G/P}$$
where the bottom row is just the natural action of $G$ on the cosets $G/P$.

The category of homogeneous vector bundles on $G/P$ is equivalent to the category $P-mod$ of representations of $P$ via $E=G \times^P \EE$, 
where $(gp,e) \simeq (g, pe)$ which surjects onto $G/P$ with fiber isomorphic to $\EE$. 
Both categories are also equivalent to the category of integral $\liep$-modules, where $\liep=\Lie P$ (see for example \cite{BK}). We indicate both the $P$-module and the $\liep$-module with the same letter. 

For any weight $\lambda$ we denote by $E_\lambda$ the homogeneous bundle corresponding to $\EE^\lambda$, the dual of the irreducible 
representation of $P$ with highest weight $\lambda$. Here $\lambda$ belongs to the fundamental Weyl chamber of the reductive part of $P$. 
Indeed, $P$ decomposes as $P=R\cdot N$ into a reductive part $R$ and a unipotent part $N$. 
At the level of Lie algebras this decomposition entails a splitting $\liep=\lier \oplus \lien$, with the obvious notation $\lier=\Lie R$ and $\lien=\Lie N$. 
By a result of Ise \cite{Ise} a representation of $\liep$ is completely reducible if and only if
it is trivial on $\lien$, hence it is completely determined by its restriction to $\lier$.

The cohomology of irreducible bundles can be computed using the Borel-Weil-Bott theorem. We recall here the result as stated by Kostant in \cite{Ko}.

Let $\Lambda$ be the fundamental Weyl chamber of $\lieg$, and let $\Lambda^+$ be the Weyl chamber of $\lier$, the reductive part
of $\lier$ in the Levi decomposition. Now consider the Weyl group $W$, and take those elements sending $\Lambda$ in $\Lambda^+$:
$$W^1=\{w \in W| w\Lambda \subset \Lambda^+\}.$$
Recall that $W$ is generated by elements of type $r_\alpha$ (where
with $r_\alpha$ we denote the reflection with respect to the
hyperplane $H_\alpha$, orthogonal to the root $\alpha$) and define
the length of an element $w \in W$, $l(w)$, as the minimum number of reflections needed to obtain $w$.
Finally, recall the affine action of the Weyl group on the weights. For $w \in W$:
$$w \cdot \lambda := w(\lambda+g)-g,$$
where $g=\sum_{i=1}^n \lambda_i$ is the sum of all fundamental weights. 
 
\begin{thm}\label{teoremaBott} Let $E_\lambda$ be an irreducible homogeneous vector bundle on $G/P$, where $\lambda \in \Lambda^+$. Then there
exists a unique element of the Weyl group $w \in W$ such that $w(\lambda+g) \in \Lambda$  (note that $w^{-1} \in W^1$).
\begin{enumerate}
\item If $w(\lambda + g)$ belongs to the interior of $\Lambda$, then setting $\nu=w\cdot \lambda$ 
we have $\HH^{l(w)}(E_\lambda)=\Sigma^\nu$, the dual of the irreducible $\lieg$-module with highest weight $\nu$, and $\HH^j(E_\lambda)=0$ for $j \neq l(w)$.
\item If $w(\lambda + g)$ belongs to the boundary of $\Lambda$, then $\HH^j(E_\lambda)=0$ for all $j$.
\end{enumerate}
\end{thm}

Let now $\xi_1,\ldots,\xi_m$, $m=\dim X$, be the weights of the
adjoint representation (which corresponds to the cotangent bundle
$\Omega^1_X$). Let $Y_{\xi_j}=H_{\xi_j}-g$, and denote by
$s_j$ the reflection through $Y_{\xi_j}$, for $j=1,\ldots,m$. For any weight $\lambda$ we have that
$$s_j(\lambda)=r_{\xi_j}(\lambda+g)-g,$$
hence if $w=r_{\xi_1}\cdot\ldots\cdot r_{\xi_p}$ then
$w(\lambda+g)-g=s_1\cdot\ldots\cdot s_p(\lambda)$.

\begin{defn}[Regular weights and Bott chambers]
A dominant weight $\nu$ is called regular if $(\nu,\alpha) \neq 0$ for every root $\alpha$, where $(\:\:,\:)$ is the Killing form.
Otherwise $\nu$ is called singular and it belongs to some hyperplane $H_\alpha$. Set:
$$\Lambda^+_0:=\{\lambda \in \Lambda^+ \:|\: \lambda+g \:\hbox{is regular}\}.$$
$\Lambda^+_0$ is divided into several ``chambers'', that we call \emph{Bott chambers}.
\end{defn}

Notice that $\Lambda^+_0$ is obtained from $\Lambda^+$ by removing
exactly the $Y_{\xi_j}$. Hence a convenient composition of the
$s_j$'s  brings $\Lambda$ exactly into the above defined Bott
chambers.\\ The length of the Weyl elements needed to take weights belonging to the same Bott chamber into the dominant chamber is constant, 
and we call it the \emph{length of the Bott chamber}. Moreover, if two Bott chambers have a common hyperplane in their boundary, then their lengths are consecutive integers.

\subsection{Definition of the quiver $\Q_X$ and its representations}\label{sezione def of the quiver}
To any rational homogeneous variety $X=G/P$ we associate a quiver with relations, that we denote by $\Q_X$. 
The idea is to exploit all the information given by the choice of the parabolic subgroup $P$, with its decomposition $P=R\cdot N$. 
For basics on quiver theory we refer the reader to \cite{DeWe}.

\begin{defn}[The quiver $\Q_X$]\label{defQX}
Given $X=G/P$, the quiver $\Q_X$ is constructed as follows. The vertices are the irreducible homogeneous bundles 
$E_\lambda$ on $X$, which we identify with elements $\lambda \in \Lambda^+$.\\
There is an arrow connecting the vertices $E_\lambda$ and $E_\mu$ if and only if the vector space $\Hom(\Omega^1_X \otimes E_\lambda,
E_\mu)^G$ is non-zero (see next Lemma \ref{key lemma}).\\
The ideal of relations on $\Q_X$ will be defined in Section \ref{sezione equivalenza categorie}.
\end{defn}

\begin{rem}\label{quante frecce}
Definition \ref{defQX} is precisely the original one of Bondal and
Kapranov \cite{BK}, later used also by Alvarez-C\'{o}nsul and
Garc\'{\i}a-Prada \cite{ACGP}. Arrows (modulo translation) correspond to weights of
the nilpotent algebra $\lien$, considered as an $\lier$-module with the adjoint action.

In fact one could obtain an equivalent theory by considering the
same vertices with a smaller number of arrows, i.e. by taking
only weights of the quotient $\lien/[\lien,\lien]$. This is for example the choice made by Hille \cite{Hi1}.
If that's the case, then the other arrows turn out to be a consequence of the relations of the quiver. This will be clarified in Section \ref{sezione equivalenza categorie}, where we define the relations with some details.
We call the first ones \emph{generating arrows}, and the other ones \emph{derived arrows}. Also, 
in naming the weights of $\lien$ by $\{\xi_1,\ldots,\xi_n\}$ as above, we do it so that the first $\ell$ ones correspond to the 
generating elements belonging to $\lien/[\lien,\lien]$ (see Lemma \ref{Prop 6.4}).\\
In the Hermitian symmetric case the two definitions agree, and they coincide with the definition of the arrows given in \cite{OR}.
\end{rem}

Let now $E$ be a homogeneous vector bundle over $X$: we want to associate to it a representation of the quiver $\Q_X$, that we call $[E]$.

The bundle $E$ comes with a filtration:
\begin{equation}\label{filtrazione}
0 \subset E_1 \subset E_2 \subset \ldots \subset E_k=E,
\end{equation}
where each $E_i/E_{i-1}$ is completely reducible. We define $\gr \EE=\oplus_i \EE_i/\EE_{i-1}$ for any filtration
(\ref{filtrazione}). The graded $\gr \EE$ (and its associated completely reducible graded vector bundle $\gr E$) do not depend on the filtration and $\gr \EE$ is given by looking at our $\liep$-module $\EE$ as a module over
$\lier$, so that it decomposes as a direct sum of irreducibles. We thus have the decomposition:
\begin{equation}\label{grE}
\gr E = \bigoplus_\lambda E_\lambda \otimes V_\lambda,
\end{equation}
with multiplicity spaces $V_\lambda \simeq \C^k$. 

The functor $\EE \mapsto \gr \EE$ from $P-mod$ to $R-mod$ (which in literature is often denoted by $Ind_R^P$) is exact.

Roughly speaking, since the information on the vertices of the quiver is encoded in $R$ and the one on the arrows in $N$, we need to identify these two 
``components'' in the bundle $E$ in order to construct $[E]$. 
There is an equivalence of categories between $\liep$-modules $\EE$ and pairs $(\FF,\theta)$, 
where $\FF$ is an $\lier$-module and $\theta: \lien \otimes \FF \rightarrow \FF$ is an equivariant morphism satisfying a certain condition. 
This equivalence is entailed by the following result:

\begin{thm}\cite[Theorem 3.1]{Bora}\label{theta} Consider $\lien$ as an $\lier$-module with the adjoint action.
\begin{enumerate}
 \item Given a $\liep$-module $\EE$, the action of $\lien$ over $\EE$ induces a morphism of $\lier$-modules $\theta: \lien \otimes \gr \EE \rightarrow \gr \EE$.
The morphism
$$\theta \wedge \theta:\wedge^2\lien\otimes \gr \EE \rightarrow \gr \EE $$
defined by $\theta \wedge \theta ((n_1\wedge n_2)\otimes f):= n_1 \cdot (n_2 \cdot f) -n_2 \cdot (n_1 \cdot f)$
satisfies the equality $\theta \wedge \theta=\theta \varphi$ in $\Hom(\wedge^2\lien\otimes \gr \EE, \gr \EE)$,
where $\varphi$ is:
  \begin{align}
  \nonumber \varphi:\wedge^2 \lien \otimes \gr \EE &\rightarrow \lien \otimes \gr \EE\\
  \nonumber (n_1 \wedge n_2)\otimes e &\mapsto [n_1,n_2] \otimes e.
  \end{align}
  \item Conversely, given an $\lier$-module $\FF$ and a morphism of $\lier$-modules $\theta: \lien \otimes \FF \rightarrow \FF$
  such that $\theta \wedge \theta=\theta \varphi$, $\theta$ extends uniquely to an action of $\liep$ over $\FF$,
  giving a $\liep$-module $\EE$ such that $\gr \EE=\FF$ (and thus a vector bundle $E$ on $X$ such that $\gr E=F$).
\end{enumerate}
\end{thm}

The decomposition (\ref{grE}) entails the other decomposition:
\begin{equation}\label{decomposizione theta}
\theta \in \Hom(\lien \otimes \gr \EE, \gr \EE)= \bigoplus_{\lambda, \mu
\in \Q_0} \Hom(V_\lambda,V_\mu) \otimes
 \Hom(\lien \otimes \EE^\lambda, \EE^\mu).
\end{equation}

Before we can give the construction of the representation $[E]$ of the quiver $\Q_X$, we need the following multiplicity result:

\begin{lem}\label{key lemma}\cite[Proposition 2]{BK}
When $G$ is of type $ADE$ the dimension $\dim \Hom(\lien \otimes \EE^\lambda, \EE^\mu)^P=\dim \Hom(\Omega^1_X \otimes E_\lambda, E_\mu)^G$ is
either 0 or 1 for every pair $\lambda, \mu \in \Lambda^+$.
\end{lem}

\begin{defn}[The representation ${[E]}$]\label{def rappresentazione}
For any $\lier$-dominant weight $\lambda$ fix a maximal vector $v^\lambda$ of $\EE^\lambda$ (it is unique up to scale). 
For any weight $\alpha$ of $\lien$, fix an eigenvector $e_\alpha \in \lien$.
Now suppose that there is an arrow $E_\lambda \rightarrow E_\mu$ in the quiver. Then the vector space $\Hom(\lien \otimes
\EE^\lambda,\EE^\mu)^P$ is non-zero, and in particular is 1-dimensional. Notice that by definition, being given by the action of
$\lien$, the arrow will send the weight $\lambda$ into a weight $\mu=\lambda+\alpha$, for some negative root 
$\alpha$ (for we have $\lieg_\alpha \cdot W_\lambda \subseteq W_{\lambda+\alpha}$).\\
Then fix the generator $f_{\lambda\mu}$ of $\Hom(\lien \otimes
\EE^\lambda,\EE^\mu)^P$ that takes $e_\alpha \otimes v^\lambda \mapsto v^\mu$.\\
Once all the generators are fixed, from (\ref{decomposizione theta}) write the map $\theta$ uniquely as:
\begin{equation}\label{decomposizione theta bis}
	\theta= \sum_{\lambda, \mu}g_{\lambda\mu}f_{\lambda\mu},
\end{equation}
and thus we can associate to the arrow $\lambda
\xrightarrow{f_{\lambda\mu}} \mu$ exactly the element
$g_{\lambda \mu}$ in $\Hom(V_\lambda,V_ \mu)$.\\ 
All in all: to the homogeneous vector bundle $E$ on $X$ we associate a representation $[E]$ of the quiver $\Q_X$ as follows. To the vertex $\E_\lambda$ we associate the vector space $V_\lambda$ from the decomposition (\ref{grE}). To an arrow $E_\lambda \rightarrow E_\mu$ we associate the element $g_{\lambda \mu} \in \Hom(V_\lambda,V_\mu)$ from the decomposition
(\ref{decomposizione theta}).
\end{defn}

A different choice of constants would have led to an equivalent construction. Moreover, the correspondence $E \mapsto [E]$ is functorial. 
For details, see \cite{OR} and \cite{ACGP}.

\subsection{Second equivalence of categories}\label{sezione equivalenza categorie}
From Proposition \ref{theta} it is clear that by putting the appropriate relations on the quiver, namely the equality 
$\theta \wedge \theta=\theta \varphi$, we can get the desired equivalence of categories. We give here a sketch of how one can derive these relations. For details, we refer the reader to \cite{Bora} and \cite{ACGP}.

Let $\lambda, \mu, \nu$ be any three vertices of the quiver $\Q_X$. What we need to do is translate the condition $\theta \wedge \theta=\theta \varphi$ in the ``language'' or the arrows of the quiver. To do this one defines an invariant morphism $\phi_{\lambda\mu\nu}$:
$$\Hom(\lien \otimes \EE^\lambda,\EE^\mu)^P \otimes \Hom(\lien \otimes \EE^\mu, \EE^\nu)^P \xrightarrow{\phi_{\lambda\mu\nu}}
\Hom(\wedge^2 \lien \otimes \EE^\lambda, \EE^\nu)^P$$
via contraction and wedge. 

In particular once the choice of scale has been made, there are fixed generators $f_{\lambda\mu}$, where
$f_{\lambda\mu}: n_\alpha \otimes v_\lambda \mapsto v_\mu$, and $\alpha=\lambda -\mu$. Then if we set $\beta=\mu -\nu$, all in all:
$$\phi_{\lambda\mu\nu}(f_{\lambda\mu} \otimes f_{\mu\nu}): (n_\alpha \wedge n_\beta) \otimes v_\lambda \mapsto v_\nu.$$

The natural morphism $\wedge^2 \lien \rightarrow \lien$ sending $n \wedge n' \mapsto [n,n']$ induces a morphism $\phi_{\lambda\nu}$:
$$\Hom(\lien \otimes \EE^\lambda,\EE^\nu)^P \xrightarrow{\phi_{\lambda\nu}} \Hom(\wedge^2 \lien \otimes \EE^\lambda, \EE^\nu)^P.$$

Theorem \ref{theta} together with the splitting (\ref{decomposizione theta bis}) entail that we have an equality in $\Hom(\wedge^2 \lien \otimes \EE^\lambda,\EE^\nu)^P$:
\begin{equation}\label{relazione da espandere}
\sum_{\lambda,\nu}\sum_\mu\Big(\phi_{\lambda\mu\nu}(f_{\lambda\mu} \otimes f_{\mu\nu})(g_{\lambda\mu}g_{\mu\nu})
+\phi_{\lambda\nu}([f_{\lambda\mu},f_{\mu\nu}])g_{\lambda\nu}\Big)=0.
\end{equation}

By expanding equality (\ref{relazione da espandere}) in a basis one gets a system of equations that the maps $g_{\gamma\delta}$ must satisfy, for every pair of vertices $\lambda,\nu$ and by imposing these same equations on the arrows of the quiver one gets the desired equivalence of categories.


In Section \ref{esempi} we give details for the case $\SL_3/B$. We underline the fact that in general the problem of finding an explicit description 
of the relations is still open. They are known on $G/B$ \cite{ACGP}; Hille computed the relations for $\PP^2$ \cite{Hi1}, and 
Ottaviani and Rubei extended them to all Grassmannians using Olver maps \cite{OR}.

\begin{thm}\cite{BK, Hi1, ACGP}\label{BK}
Let $X$ a rational homogeneous variety of type $ADE$. The category of integral $\liep$-modules is equivalent to the category of finite dimensional
representations of the quiver $\Q_X$ with the relations defined in Section \ref{sezione equivalenza categorie}, and it is equivalent to the category of homogeneous
vector bundles on $X$.
\end{thm}

\section{Computing sections}\label{parte nuova}

\subsection{Preliminaries} 

In all this section $X=G/P$ is a rational homogeneous variety, and $G$ is a complex semisimple Lie group of type $ADE$.

We start with a definition.
\begin{defn}\cite[Def. 5.10]{OR}\label{def sottofibrato}
Let $E$ be a homogeneous vector bundle on $X$, $=\gr E=\bigoplus_{i=1}^n V_{\lambda_i} \otimes E_{\lambda_i}$, so that 
$V=\bigoplus_{i=1}^n V_{\lambda_i}$ is a $\C\Q_X$-module. (For the sake of simplicity, we denote by $\C\Q_X$ the path 
algebra of the quiver with relations $\Q_X$, meaning that the algebra has been divided by the ideal of relations.)\\ 
For any subspace $V'\subseteq V$, $V'=\bigoplus_{j\in J \subseteq \{1,\ldots,n\}} V_{\lambda_j}$, 
the submodule generated by $V'$ defines a homogeneous subbundle of $E$. In a similar fashion, let $(V':\C\Q_X)=\{v \in V\:|\:fv \in
V',\:\forall\:f \in \C\Q_X\}$, which is a submodule, then the quotient
$V/(V':\C\Q_X)$ defines a homogeneous quotient of $E$. 
\end{defn}

\begin{lem}[Generalization of Prop. 6.4, \cite{OR}]\label{Prop 6.4}
Let $\lambda, \mu \in \Lambda^+$ be in two adjacent Bott chambers with
$\HH^i(E_\lambda) \simeq H^{i+1}(E_\mu)$ isomorphic as $G$-modules. Then $\mu-\lambda=k \xi_j$ for some $k \in \Z^+$ and for some weight $\xi_j$ of $\lien$, and moreover:
\begin{equation}\label{al posto di symk}
\dim \Hom (E_\lambda \otimes E_{k\xi_j},E_\mu)^G=1.
\end{equation}
In particular if $\HH^0(E_\lambda) \simeq \HH^1(E_\mu)$, then the weight $\xi_j \in \{\xi_1,\ldots,\xi_\ell\}$ is a (negative) simple root.
\end{lem}

\begin{proof}
For the first part of the statement, notice that the only assumption made in \cite{OR} is that all
the roots have the same length, i.e. that the group $G$ is of type $ADE$.\\
Formula (\ref{al posto di symk}) follows from a generalization of the Pieri formula, see \cite{Lit}.\\
The second part of the statement is just a rephrasing of the fact that all length $1$ Bott chambers are separated from the dominant chamber 
by a simple reflection, by Bott theorem \ref{teoremaBott}.
\end{proof}

\begin{rem}
The word ``adjacent'' in the statement of Lemma \ref{Prop 6.4} means that the two Bott chambers are symmetric
with respect to an hyperplane $Y_j$ orthogonal to one of the roots, or in other words that the reflection $s_j$ exchanges them. 
We take into account all weights, not just
the generating ones. Thus the hierarchy between generating and derived arrows of the quiver 
translates in this setting in two ways of being adjacent for the chambers, depending on
whether or not the hyperplane with respect to which we are
reflecting is orthogonal to an element of $\lien/[\lien,\lien]$. 
\end{rem}

\subsection{$\mathbf{A_m}$-type bundles}\label{spaghetti}

In order to compute sections of homogeneous bundles using quiver representations, 
we first need to deal with cohomology of a special type of bundles, that we call ``$A_m$-type bundles''. Such bundles 
have convenient splitting properties entailed by quiver theory.

\begin{defn}\label{def spaghetto}
Let $E$ be a homogeneous bundle on $G/P$. Let $\Q|_E$ denote the subquiver of $\Q_{G/P}$ given by all vertices where 
$[E]$ is non-zero and all arrows connecting any two such vertices.
The support of $\Q|_E$ has at most $\rk(E)$ vertices. The representation $[E]$ of $Q_{G/P}$ induces a representation of the subquiver $\Q|_E$.\\
We call $E$ an \emph{$A_m$-type bundle} if the quiver $\Q|_E$ is of type $A_m$, i.e. if in the decomposition $\gr E=\oplus E_\lambda \otimes V_\lambda$, $V_\lambda$ is zero outside a path connecting the vertices $\{\lambda + p\xi_j\:|\: 0 \leq p \leq k \}$.
\end{defn}

\begin{rem} As a consequence of the relations of the quiver and of the equivalence of categories, every $A_m$-type bundle in the direction of derived arrows is completely reducible.
\end{rem}

Moreover, if $S$ is an $A_m$-type bundle, then for the representation $[S]$ holds the following well-known theorem, see Gabriel [GR].
\begin{thm}[Gabriel]\label{Gabriel}
Every representation of the $A_m$ quiver is the direct sum of irreducible representations with dimension vector
$$(0,0,\ldots,0,1,1,\ldots,1,0,\ldots,0),$$
where the nontrivial linear maps are isomorphisms. In particular if the direction is along a derived arrow, there is just one $1$.
\end{thm}

\begin{thm}\label{teorema per gli spaghetti}
Let $S$ be an $A_m$-type bundle on $X$ with $\HH^0(\gr S) \neq 0$. Let $E_\lambda$ be an irreducible summand of $\gr S$ such that $\HH^0(E_\lambda)=V$, for $V=\Sigma^\lambda$ nonzero irreducible $G$-module. Then:
\begin{enumerate}
	\item If among the summands of $\gr S$ there is an $E_\mu$ such that $\HH^1(E_\mu)=V$, consider in the quiver $\Q|_S$ the path from $E_\lambda$ to $E_\mu$. By composing the linear maps corresponding to this path in the representation $[S]$, we get a linear map $\sigma_0^V: V_\lambda \rightarrow V_\mu$. Then the isotypical component $\HH^0(S)^V=V \otimes (\ker \sigma_0^V)$. 
	\item If there is no such $E_\mu$, then $\HH^0(S)^V= \HH^0(E_\lambda \otimes V_\lambda)=V^{\oplus \dim V_\lambda} $. 
\end{enumerate}
\end{thm}

\begin{proof} Using Theorem \ref{Gabriel}, the only non-trivial case is when the quiver $\Q|_S$ connects a dominant weight $\lambda$ 
with its image under a simple reflection in one of the length $1$ Bott chambers. In this case Lemma \ref{Prop 6.4} applies, and 
$\mu=s_j(\lambda)$ is such that $\mu-\lambda=k\xi_j$ for a generating element $\xi_j$. 

Theorem \ref{Gabriel} also implies that there is no loss in generality in assuming that all multiplicity spaces are 1-dimensional. 

If along the path from $\lambda$ to $\mu$ one of the maps in the representation $[S]$ is zero, then the linear map $V_\lambda \rightarrow V_\mu$ is zero as well, and so is the map $ V \xrightarrow{\sigma_0^V =0} V$. If this is the case, the bundle $S$ splits as $S = S_1 \oplus S_2$, where $\HH^1(S_1)^V=\HH^0(S_2)^V=V$ and $\HH^0(S_1)^V=\HH^1(S_2)^V=0$. Then we have $\HH^0(S)^V=V=\ker \sigma_0^V$.\\

Suppose now that all the maps in $[S]$ corresponding to the path from $\lambda$ to $\mu$ are nonzero. We need to show that in this case $S$ has vanishing cohomology.

\begin{lem}\label{A singolare}
Let $\lambda'$ (respectively $\mu'$) be the vertex of the Bott chamber containing $\lambda$ (respectively $\mu$). Under our assumption 
$\lambda'=0$ is the zero weight. Consider the irreducible bundles $E_{\lambda'}=\OO_X$ and $E_{\mu'}$. Let $A$ be the unique (up to scale) 
indecomposable bundle in the extension:
$$0 \rightarrow E_{\mu'} \rightarrow A \rightarrow \OO_X \rightarrow 0.$$
Then all the cohomology of $A$ vanishes.
\end{lem}
\begin{proof} 
Suppose $\HH^0(A)=\Hom(\OO,A) \neq 0$. Then there is a $G$-invariant morphism $\OO_X \xrightarrow{\varphi} A$. 
In particular it is a morphism $\varphi: [\OO_X] \rightarrow [A]$ between representations of the same quiver. 
Notice that the subquiver $\Q|_A$ simply consists of the two vertices $\OO_X$ and $E_{\mu'}$ connected by the arrow corresponding to the weight $\xi_j$ defined above:
$$\Q|_A \:\: =\:\: \OO_X \rightarrow E_{\mu'}$$
The representation of $\Q|_A$ associated to the irreducible bundle $\OO_X$ is:
$$[\OO_X] \:\: =\:\: \C \rightarrow 0$$
The representation of $\Q|_A$ associated to the bundle $A$ is:
$$[A] \:\: =\:\: \C \xrightarrow{a\cdot \Id} \C$$
where $a$ is the non-zero constant determining the element $A$ in the space $\Ext^1(\OO_X,E_{\mu'})^G=\C$.
Any such morphism $\varphi: [\OO_X] \rightarrow [A]$ is forced to be zero by the very definition of morphism of quiver representation. 
\end{proof}

\begin{prop}\cite[Prop 6.8]{OR}\footnote{The proof in \cite{OR} applies to the general case with the only caveat that we need to take only the 
	maximal weights of $\gr H$}.\label{Prop 6.8}
Let $\nu$ be a weight such that $\HH^1(E_\nu)$ is some $G$-module $\Sigma \neq 0$. Let $\nu'$ be the vertex of the Bott chamber containing $\nu$. Let $H$ be the vector bundle $H= G \times^P (\EE^{\nu'} \otimes \Sigma)$. Then $E_\nu$ is the only direct summand of $\gr H$ whose cohomology in any degree is $\Sigma$.	
\end{prop}

\noindent \emph{End of proof of Theorem \ref{teorema per gli spaghetti}.} Let $A$ be the vector bundle defined in Lemma \ref{A singolare}, and say $A=G \times^P \mathcal{A}$. Let $A \otimes V$ denote the vector bundle $G \times^P (\mathcal{A} \otimes V)$.\\ 
Consider the subbundle of $A \otimes V$ generated by all direct summands (of the graded) isomorphic to $E_\lambda$, that is, take $V'=V_\lambda$ in 
Definition \ref{def sottofibrato}. Call this subbundle $K$. Then there is an exact sequence:
$$0 \longrightarrow K \longrightarrow A \otimes V \longrightarrow (A \otimes V)/K \longrightarrow 0.$$
By Proposition \ref{Prop 6.8}, both $\HH^i(K)^V$ and $\HH^i((A \otimes V)/K)^V$ are nonzero for at most $i=0,1$. Indeed, $\gr K$ contains all the direct summands of $\gr(A \otimes V)$ isomorphic to $E_\mu$. If it didn't, then $E_\mu$ would be contained in the quotient $(A \otimes V)/K$ and hence 
$\HH^1((A \otimes V)/K)^V \neq 0$. Now since $A$ is singular, so is $A \otimes V$ (because from an holomorphic point of view $\HH^i(A \otimes V)$ is just $\dim V$ copies of $\HH^i(A)$). Hence $\HH^1((A \otimes V)/K)^V \neq 0$ implies $\HH^2(K)^V \neq 0$, and this is a contradiction. Thus $\HH^i((A \otimes V)/K)^V=0$ for all $i$'s, and the same for $\HH^i(K)^V$.\\
Finally, let $K'$ be the quotient of $K$ obtained by restricting the quiver representation to the 
path joining the vertices corresponding to $E_\lambda$ and $E_\mu$. Note that the arrows of $K'$ are only generating arrows, and that's why we get a quotient. 
We have:
$$0 \rightarrow K'' \rightarrow K \rightarrow K' \rightarrow 0.$$
Now by construction the kernel $K''$ is such that $\HH^i(\gr K'')^W=0$, and hence $\HH^i(K'')^W=0$. But also $\HH^i(K)^W=0$, so the same holds for 
$K'$. By Theorem \ref{Gabriel}, $K'$ decomposes into the direct sum of several copies of the original $A_m$-type bundle $S$, by definition of $K$, hence we are done.
\end{proof}

\begin{rem}
We underline the fact that in order to compute sections of an $A_m$-type bundle $S$ it is enough to compute the maps of the associated quiver representation once and for all. In particular in many cases we do not even need to compute explicitly these maps, all we need is the information on wether these maps are zero or not.
\end{rem}

\subsection{The general case}\label{general}

Let $E$ be a homogeneous vector bundle on $X=G/P$. We want to show how to compute its sections by using 
the associated quiver representation $[E]$. This is done via the study of the map $c_0$, that we construct using the previous results on $A_m$-type bundles.

\begin{lem}\label{distinguished iso}
	Let $E$ be a homogeneous vector bundle on $X$ and let $E_\lambda$ be an irreducible summand of $\gr E$ such that $\HH^0(E_\lambda)\neq 0$. Let $E_\mu$ be such that $\HH^0(E_\lambda)=\HH^1(E_\mu)$. Then we can construct a distinguished isomorphism $j_{\lambda\mu}:\HH^0(E_\lambda) \rightarrow \HH^1(E_\mu)$.
\end{lem}

\begin{proof}
	Say $E_\mu$ irreducible summand of $\gr E$ is such that $\mu=s_j(\lambda)$ for some $j$ and $\HH^0(E_\lambda)=\HH^1(E_\mu)$. 
	Consider in the quiver $\Q_X$ the path from $E_\lambda$ to $E_\mu$:
	from Lemma \ref{Prop 6.4} we know that it is a sequence of arrows all in the  same direction $\xi_j$ (corresponding to a generating arrow).

	Let $S$ denote the homogeneous $A_m$-type bundle starting from $E_\lambda$ and ending in $E_\mu$, with the same quiver representation maps as for $[E]$. Note that $S$ might neither be a subbundle nor a quotient of $E$.\\ 
	The irreducible bundle $E_\lambda$ is a quotient of $S$ (it is a sink in the quiver $\Q|_S$):
	\begin{equation}\label{sequenza 1}	
	0 \rightarrow Z \rightarrow S \rightarrow E_\lambda \rightarrow 0,
	\end{equation}
	while the irreducible bundle $E_\mu$ injects itself in $Z$ (being a source):
	\begin{equation}\label{sequenza 2}	
	0 \rightarrow E_\mu \rightarrow Z \rightarrow Z/E_\mu \rightarrow 0
	\end{equation}

	From sequence (\ref{sequenza 1}) and Theorem \ref{teorema per gli spaghetti} we have an isomorphism:
	$$\HH^0(E_\lambda) \xrightarrow{\partial} \HH^1(Z)^W,$$
	and from sequence  (\ref{sequenza 2}) we have another isomorphism:
	$$\HH^1(E_\mu) \xrightarrow{i} \HH^1(Z)^W,$$
	and hence we get a distinguished isomorphism:
	\begin{equation}
		j_{\lambda\mu}=i^{-1} \circ \partial:\HH^0(E_\lambda) \rightarrow \HH^1(E_\mu).
	\end{equation}	
\end{proof}	
	
\begin{defn}[The morphism $c_0$]\label{def mappe c_i}
In the same setting as above, define the maps
\begin{center}
	$c_{\lambda\mu}:\HH^0(E_\lambda \otimes V_\lambda) \rightarrow \HH^1(E_\mu \otimes V_\mu)$
\end{center}
to be the tensor product of the distinguished isomorphism $j_{\lambda\mu}$ from Lemma \ref{distinguished iso} with the composition of 
the maps $V_\lambda \rightarrow V_\mu$ in the quiver representation.
Putting together all these maps for all the possible pairings $(\lambda,\mu)$ we get a map $c_0:=\sum_{\lambda,\mu}c_{\lambda\mu}$,
$$c_0:\HH^0(\gr E) \rightarrow \HH^1(\gr E).$$
\end{defn}

\begin{thm}\label{teoremone}
Let $E$ a homogeneous vector bundle on $X$, and construct the map $c_0$ as in Definition \ref{def mappe c_i}. Then $\HH^0(E)=\Ker c_0$.
\end{thm}

\begin{proof} The proof requires two steps:
\begin{itemize}
\item[1)] reduction to the case where $\HH^i(\gr E)\neq 0$ only for $i=0,1$.
\item[2)] induction on the number of irreducible summands of $\gr E$.
\end{itemize}

\underline{Step 1:} Call $E'$ the subbundle of $E$ generated---in the sense of the above Definition
\ref{def sottofibrato}---by all irreducible summands
$E_\lambda$ of $\gr E$ such that $\HH^i(E_\lambda) \neq 0$ for some
values of $i \geq 2$, hence $\HH^i(\gr E')=0$ for $i \leq 1$. We get a short exact sequence:
$$0 \rightarrow E' \rightarrow E \rightarrow E/E' \rightarrow 0.$$
The quotient $E/E'$ has the property that
$\HH^i(\gr (E/E'))=0$ for $i \geq 2$, that is any element of its graded
bundle has vanishing cohomology for $i \ge 2$, and we have
equalities:
$$\HH^0(E)=\HH^0(E/E')\:\:\:\:\:\hbox{and}\:\:\:\:\:\HH^1(\gr E)=\HH^1(\gr (E/E')).$$
Let $c_0^E:\HH^0(\gr E)\rightarrow \HH^1(\gr E)$ be the map defined in Definition \ref{def mappe c_i} for the bundle $E$. Let $c_0^{E/E'}$
the corresponding map for the quotient bundle.

Both the auxiliary $A_m$-type bundles and the distinguished isomorphisms appearing in Definition \ref{def mappe c_i} for the bundle $E$ coincide with those for $E/E'$: in other
words the maps $c_0^E$ and $c_0^{E/E'}$ are equal, and the same holds for their kernels, and this concludes Step 1.

From now on we will thus suppose, without any loss of generality, that:
$$\HH^i(E_\lambda)=0,\:\:\forall\: i \ge 2,\:\: \forall\: E_\lambda \in \gr E,$$
and we move on to Step 2.\\

\underline{Step 2:} The base step of the induction is the case where $E$ is irreducible: here we apply Theorem \ref{teoremaBott} 
and we are done.

Let us now confront the inductive step. Given our bundle $E$, consider the quiver $\Q|_E$ , that we can suppose connected. 
Choose an ordering of the vertices in a way that $ta > ha$ for every arrow
$a$. Call $E_\mu$ the irreducible bundle associated to one of the
minimal vertices of $\Q|_E$, and notice that it is a sink. Thus we
have the short exact sequence:
\begin{equation}\label{succ}
0 \rightarrow E_\mu \rightarrow E \rightarrow Q \rightarrow 0,
\end{equation}
where $Q$ is the quotient. Keeping the same notation as above, call
$c_0^E$ the linear map from Definition \ref{def mappe c_i} for the bundle $E$ and $c_0^Q$ the correspondent map for the vector bundle $Q$. 
For the bundle $Q$ the inductive hypothesis holds: $\HH^0(Q)=\Ker c_0^Q$.

We recall that $E_\mu$ has natural cohomology, and we look at the three possible situations that can occur.

If $E_\mu$ is singular, then the cohomology of $E$ and of $Q$ are trivially equal. But the singularity of $E_\mu$ also
implies that it won't appear at all in the construction of $c_0^E$.
This means that $c_0^E=c_0^Q$ and by inductive hypothesis we are done.

If $\HH^0(E_\mu)=\Sigma^\mu \neq 0$ then we have that:
$$\HH^0(E)=\HH^0(Q)\oplus \Sigma^\mu \:\:\:\:\:\hbox{and}\:\:\:\:\: \Ker c_0^E=\Ker c_0^Q\oplus \Sigma^\mu.$$
Since $\HH^0(Q)=\Ker c_0^Q$, it follows that $\HH^0(E)=\Ker c_0^E$, as we wanted.

The third situation occurs when $\HH^1(E_\mu)=W$ for some
nonzero $G$-module $W$. This time we limit ourselves to the $W$-isotypical component, for there is no loss in generality in doing it. We have equalities:
$$\HH^0(\gr E)^W=\HH^0(\gr Q)^W\:\:\:\:\:\hbox{and}\:\:\:\:\:\HH^1(\gr E)^W=\HH^1(\gr Q)^W \oplus W.$$
The map $c_0^Q$ is the composition of $c_0^E$ with the projection from $\HH^1(\gr E)^W$ on its direct summand
$\HH^1(\gr Q)^W$.

Now if the extension (\ref{succ}) splits, then we get $\HH^0(E)^W=\HH^0(Q)^W$ and $\Ker c_0^E=\Ker c_0^Q$, and we are done.

We then assume that the extension (\ref{succ}) does not split. 
The situation is best described in the following diagram:
$$\xymatrix@C-1ex{&           & & W \ar@{^{(}->}[d] \ar[dr]&         &   \\
 0 \ar[r] &(\Ker c_0^E)^W \ar[d]\ar[r]
 &\HH^0(\gr E)^W \ar@{=}[d] \ar[r]& \HH^1(\gr E)^W  \ar[r]\ar@{>>}[d]& (\coker c_0^E)^W \ar[d]\ar[r] & 0\\
 0  \ar[r]&(\Ker c_0^Q)^W \ar[r]
&\HH^0(\gr Q)^W \ar[r] & \HH^1(\gr Q)^W \ar[r] \ar[d]& (\coker c_0^Q)^W\ar[r]&0\\
&&&0&& }$$
that with a simple application of the snake lemma
yields to:
\begin{equation}\label{seconda riga}
0 \rightarrow (\Ker c_0^E)^W \rightarrow (\Ker c_0^Q)^W \rightarrow W
\rightarrow (\coker c_0^E)^W \rightarrow \ldots.
\end{equation}
On the other side we have the long exact cohomology sequence associated to (\ref{succ}):
\begin{equation}\label{prima riga}
0 \rightarrow \HH^0(E)^W \rightarrow \HH^0(G)^W \rightarrow
\HH^1(E_\mu) \rightarrow \ldots.
\end{equation}

Sequences (\ref{seconda riga}) and (\ref{prima riga}), together with the
inductive hypothesis all fit together in the following diagram:

\begin{equation}\label{diagramma c0}
\xymatrix{0\ar[r] & (\Ker c_0^E)^W \ar[r] \ar@{-->}[d]^{?}& (\Ker c_0^Q)^W \ar[d]^{\sim}\ar[r]& W \ar@{=}[d]\\
          0\ar[r] & \HH^0(E)^W \ar[r]  & \HH^0(Q)^W \ar[r]^{\partial}            & \HH^1(E_\mu)}
\end{equation}

When we prove that diagram \ref{diagramma c0} is commutative, by induction the thesis will follow. 
We can assume that there exists an 
irreducible summand $E_\lambda \in \gr Q$ such that $\HH^0(E_\lambda)=W$ (because otherwise there is nothing to prove). 
The commutativity of the diagram is entailed by the following:

\underline{Claim:} the coboundary map $\HH^0(Q)^W \xrightarrow{\partial} \HH^1(E_\mu)$ can be identified with:
$$ \HH^0(E_\lambda \otimes V_\lambda) \xrightarrow{\partial} \HH^1(E_\mu)=W,$$
where the latter is the coboundary map of the $A_m$-type bundle obtained by considering in the quiver $\Q|_E$ the path 
starting from $E_\lambda$ and ending in $E_\mu$, with the same quiver representation maps as for $E$. 
The claim reduces to the case of a coboundary map of a $A_m$-type bundle. Theorem \ref{teorema per gli spaghetti} 
then tells us that in this case if the extension doesn't split, the coboundary map is nonzero, and this concludes the proof.

\emph{Proof of claim:} Consider the vertex $E_\lambda$ together with all arrows starting from it, 
and take the subrepresentation of $[E]$ generated by these. Call $K$ the associated subbundle of $E$. We have:
$$0 \rightarrow K \rightarrow E \rightarrow E/K \rightarrow 0,$$
that fits, together with (\ref{succ}), in a diagram:
$$\xymatrix@R-1ex@C-1ex{&&0\ar[d]&0\ar[d]&\\
0\ar[r]&E_\mu\ar@{=}[d]\ar[r]&K\ar[d]\ar[r]&K/E_\mu\ar[r]\ar[d]&0\\
0\ar[r]&E_\mu\ar[r]&E\ar[d]\ar[r]&E/E_\mu=Q\ar[r]\ar[d]&0\\
&&E/K\ar[d]\ar@{=}[r]&E/K\ar[d]&\\
&&0&0&\\}$$

Taking long cohomology sequences associated with the diagram above, the two coboundary maps:
$$\xymatrix{\HH^0(K/E_\mu)^W \ar[d]\ar[r]^-{\partial}&\HH^1(E_\mu) \ar@{=}[d]\\
\HH^0(Q)^W \ar[r]^-{\partial}&\HH^1(E_\mu)}$$
can be identified, since the above square commutes.\\
Now we repeat the argument: let $S$ be the $A_m$-type bundle whose associated representation of the quiver 
is the quotient subrepresentation of $K$ obtained by taking only the $A_m$-type path that takes from $E_\lambda$ to $E_\mu$. 
Again, we have a short exact sequence:
$$0 \rightarrow \mathcal{K} \rightarrow K \rightarrow S \rightarrow 0,$$
and a diagram (note that $E_\mu \hookrightarrow S$):
$$\xymatrix@R-1ex@C-1ex{&&0\ar[d]&0\ar[d]&\\
&&E_\mu \ar[d]\ar@{=}[r]&E_\mu\ar[d]&\\
0\ar[r]&\mathcal{K}\ar@{=}[d]\ar[r]&K\ar[d]\ar[r]&S\ar[r]\ar[d]&0\\
0\ar[r]&\mathcal{K}\ar[r]&K/E_\mu\ar[d]\ar[r]&S/E_\mu\ar[r]\ar[d]&0\\
&&0&0&\\}$$
From taking long cohomology sequences associated with this second diagram, we get the identification of the two coboundary maps:
$$\xymatrix{\HH^0(K/E_\mu)^W \ar[d]\ar[r]^-{\partial}&\HH^1(E_\mu) \ar@{=}[d]\\
\HH^0(S/E_\mu)^W \ar[r]^-{\partial}&\HH^1(E_\mu)}$$
and the claim is proved, and so is the statement of Theorem \ref{teoremone}.
\end{proof}

\begin{rem}
One of the advantages of our construction is that one only has to deal with maps $\HH^0(\gr E) \rightarrow \HH^1(\gr E)$, whereas 
when using spectral sequences it is often necessary to compute $\HH^2(\gr E)$ and the related maps as well. 
\end{rem}

\section{Remarks and examples}\label{esempi}

A natural question arising is wether our construction can be simplified in any way. Suppose we have a homogeneous bundle $E$ with two irreducible summands $E_\lambda$ and $E_\mu$ of $\gr E$ such that $\HH^0(E_\lambda)=\HH^1(E_\mu)$. One wonders if it is possible to avoid the use of the distinguished isomorphism in Definition \ref{def mappe c_i} and simply compose the linear maps corresponding to the path from $E_\lambda$ to $E_\mu$ in the representation $[E]$, just like it is done in Theorem \ref{teorema per gli spaghetti} for $A_m$-type bundles.

If we could avoid the distinguished isomorphism, then we could also extend the result to higher cohomology. From Proposition \ref{Prop 6.4} we know that even in this more general situation $\mu-\lambda=k \xi_j$ for some $k \in \Z^+$ and for some weight $\xi_j$ of $\lien$, and that $\dim \Hom (E_\lambda \otimes E_{k\xi_j},E_\mu)^G=1$. Of course, if the weights $\lambda$ and $\mu$ belong to two ``strictly'' adjacent Bott chambers the path connecting them is in fact the support of some nonzero $A_m$-type bundle. If instead we deal with a derived arrow, every $A_m$-type bundle in that direction is completely reducible, yet the path connecting these two weights is well-defined, and we could still try to use the ``naive'' definition of just composing the linear maps in the representation. We would then get maps $c_i:\HH^i(\gr E) \rightarrow \HH^{i+1}(\gr E)$, just like in \cite{OR}. 

Unfortunately this is not the case. Let us look at some examples to see what happens.\\

We now work on the flag $\F=\SL_3/B$. We start by describing with some details the quiver and the relations for this variety.
The $\lier$-dominant weights are $\Lambda^+=\Z^3$, triplets $(\lambda_1,\lambda_2,\lambda_3)$ with the condition $\lambda_1+\lambda_2+\lambda_3=0$. Irreducible bundles on $\F$ are the line bundles:
\begin{equation}\label{iso per il caso 3-fold}
E_\lambda=E_{(\lambda_1,\lambda_2,\lambda_3)}=\OO_\F(\lambda_1-\lambda_2,\lambda_2-\lambda_3).
\end{equation}
The nilpotent algebra $\lien$ has three weights $\beta_1=(-1,1,0)$, $\beta_2=(0,-1,1)$ and $\beta_1+\beta_2=(-1,0,1)$. Call $e_1$, $e_2$ and $e_{12}$ the corresponding directions of the arrows.

In the Borel case, the relations of the quiver given in Section \ref{sezione equivalenza categorie} have an explicit form. 
Namely, for each root $\alpha$, let $e_\alpha \in \lieg_\alpha$ be the corresponding Chevalley generator, 
and define the Chevalley coefficients $N_{\alpha\beta}$ by setting $[e_\alpha,e_\beta]=N_{\alpha\beta}e_{\alpha+\beta}$, if
$\alpha+\beta \in \Phi^-$, and $N_{\alpha\beta}=0$ otherwise. 

\begin{prop}\cite[Proposition 1.21]{ACGP}\label{def relazioni bis}
The relations $\mathcal{R}$ on the quiver $\Q_{G/B}$ are the ideal
generated by all the equations:
$$\mathcal{R}_{(\alpha,\beta)}=\{f_{\lambda\mu}f_{\mu\nu}-f_{\lambda\mu'}f_{\mu'\nu}-N_{\alpha\beta}g_{\lambda\nu}=0\}$$
for $\alpha < \beta \in \Phi^-$ and for any pair of weights
$\lambda,\nu$, where $\alpha+\beta=\lambda-\nu$,
$\mu=\lambda+\alpha$ and $\mu'=\lambda+\beta$.
\end{prop}

In our example we get the Serre relations:
\begin{align}
&\mathcal{R}_{(\beta_1,\beta_2)}=\{[e_1,e_2]=e_{12}\} \label{relf1}\\
&\mathcal{R}_{(\beta_1,\beta_{12})}=\{[e_1,e_{12}]=0\} \label{relf2}\\
&\mathcal{R}_{(\beta_2,\beta_{12})}=\{[e_2,e_{12}]=0\}\label{relf3}
\end{align}

The three directions of the arrows send a vertex $\OO_\F(a,b)$ to:
$$\xymatrix{\OO_\F(a-2,b+1)&\OO_\F(a-1,b-1)\\
\OO_\F(a,b) \ar[u]^{e_1} \ar@{-->}[ur]^{e_{12}} \ar@{~>}[r]^{e_2}& \OO_\F(a+1,b-2)}$$

Consider now the following subquiver of $\Q_\F$:
\begin{equation}\label{subquiver per esempi}
\xymatrix{\OO_\F(-4,2) \ar@{~>}[r]& \OO_\F(-3,0)\\
\OO_\F(-2,1) \ar[u]\ar@{~>}[r]\ar@{-->}[ur]& \OO_\F(-1,-1)\ar[u]\\
\OO_\F \ar[u] \ar@{-->}[ur]}
\end{equation}

The dashed arrow is the bracket of the other two, because $e_{13}=[e_{12},e_{23}]$. 
Also $\HH^0(\OO_\F)=\HH^1(\OO_\F(-2,1))=\HH^2(\OO_\F(-3,0))=\C$ and both $\OO_\F(-4,2)$ and $\OO_\F(-1,-1)$ are singular.

\begin{ex}\label{esempio A} Let $A$ be the rank 3 vector bundle on $\F$ constructed as follows. First take a non-split element $A_1$  (unique up to scale):
\begin{equation}\label{estensione A1}
	0 \rightarrow \OO_\F(-1,-1) \rightarrow A_1 \rightarrow \OO_\F(-2,1) \rightarrow 0,
\end{equation}
and then define $A$ to be the unique (again, up to scale) non-split extension:
\begin{equation}\label{estensione A}
	0 \rightarrow \OO_\F(-3,0) \rightarrow A \rightarrow A_1 \rightarrow 0.
\end{equation}
Since $\OO_\F(-1,-1)$ is singular, we get that $\HH^i(A_1)=\HH^i(\OO(-2,1))$ is nonzero only for $i=1$ and $\HH^1(A_1)=\HH^1(\OO(-2,1))=\C$.\\ 
Moreover $\HH^2(\OO(-3,0))=\HH^1(\OO(-2,1))=\C$, and hence if we look at the long cohomology sequence associated with (\ref{estensione A}) we get:
\begin{equation}\label{cohomology A}
	0 \rightarrow \HH^1(A) \rightarrow \HH^1(\OO_\F(-2,1)) \xrightarrow{\partial} \HH^2(\OO_\F(-3,0)) \rightarrow \HH^2(A) \rightarrow 0.
\end{equation}
We claim that the coboundary map $\partial$ is nonzero and thus $A$ is singular. Suppose not; then $\HH^1(A)=\HH^2(A)=\C$, and in particular there 
would be a non-split extension defining a vector bundle $L \in \Ext^1(\OO_\F,A)$:
\begin{equation}\label{estensione L}
	0 \rightarrow A \rightarrow L \rightarrow \OO_\F \rightarrow 0.
\end{equation}
Let's look at the representation $[L]$ of the quiver $\Q_\F$ associated with the bundle $L$.
For the sake of simplicity, we are going to look at it as a representation of the subquiver of $\Q_\F$ that we drew in (\ref{subquiver per esempi}). 
From (\ref{estensione A1}), (\ref{estensione A}) and (\ref{estensione L}) we get:
$$[L] = \xymatrix@R5ex@C7ex{0\ar@{~>}[r]& \C\\
\C\ar[u]\ar@{~>}[r]_{a_1}\ar@{-->}[ur]_-{-a_2a_1}&\C\ar[u]_{a_2}\\
\C\ar[u]^{\ell} \ar@{-->}[ur]_{\ell a_1}}$$
For the relations to be satisfied the constant $\ell \neq 0$ giving $L$ as non-split extension has to be such that $\ell a_1 a_2 =-\ell a_1 a_2$. 
If we assume that both $a_1$ and $a_2$ are non zero, i.e. that neither of the extensions 
(\ref{estensione A1}) and (\ref{estensione A}) split, this forces $\ell=0$, and our claim is proved.
\end{ex}

\begin{ex}\label{esempio B}
As a second example, we now study the vector bundle $B$ on $\F$ constructed as follows. 
Take the same subquiver of $\Q_\F$ drawn in (\ref{subquiver per esempi}). To fix ideas, now take the vector bundle $A$ to be the one constructed in Example \ref{esempio A} where this time we fix both the constants $a_1=a_2=1$. Extend this to an element of $\Ext^1(A,\OO_\F(-4,2))$. Notice that we have $\Ext^1(A,\OO_\F(-4,2))=\C^2$ so the extensions are parametrized by a pair: \begin{center}
$(s,t) \in \Ext^1(\OO_\F(-2,1),\OO_\F(-4,2)) \times \Ext^1(\OO_\F(-3,0),\OO_\F(-4,2))$.
\end{center} 
Call the vector bundle in the extension $B_{s,t}$. 
We now let $t=1$ and extend the vector bundle $B_{s,1}$ with the irreducible bundle $\OO_\F$. We call $F$ the element in the extension $\Ext^1(\OO_\F,B_{s,1})=\C$ given by the constant $f \neq 0$: 
\begin{equation}\label{estensione F}
 	0 \rightarrow B_{s,1} \rightarrow F \rightarrow \OO_\F \rightarrow 0.
 \end{equation}
It is easy to check that the only value of $s$ such that the bundle $B_{s,1}$ does extend to a bundle $F$, no matter for what value of $f$, is $s=2$. 
We fix $f=1$, 
and we compute the cohomology of $F$. Using similar techniques to the previous Example \ref{esempio A}, one gets that $\HH^0(F)=\HH^1(F)=0$ and $\HH^2(F)=\C$.

This allow us to compute the cohomology of $B_{2,1}$. From sequence (\ref{estensione F}) we get that $\HH^1(B_{2,1})=\HH^2(B_{2,1})=\C$. 
\end{ex}
What do Examples \ref{esempio A} and \ref{esempio B} tell us? Let's take a second look at the (representations of the quiver associated) to the two bundle $A$ and $B=B_{2,1}$. 
$$[A] = \xymatrix@R5ex@C7ex{0\ar@{~>}[r]& \C\\
\C\ar[u]\ar@{~>}[r]_1\ar@{-->}[ur]_{-1}&\C\ar[u]_1\\
0\ar[u] \ar@{-->}[ur]}\:\:\:\:\:\:\:\:\:\:\:\:\:\:\:\:
[B] = \xymatrix@R5ex@C7ex{\C\ar@{~>}[r]^1& \C\\
\C\ar[u]^2\ar@{~>}[r]_1\ar@{-->}[ur]_1&\C\ar[u]_1\\
0\ar[u] \ar@{-->}[ur]}$$

The picture above shows that the naive definition that we wanted to try above is indeed impossible. In the two situations $A$ and $B$ have the same 
graded bundle up to a singular factor, and the map $c_1$ would be the same (up to sign). Yet in the first situation the coboundary map is 
nonzero and the bundle $A$ is singular, while in the second situation the coboundary map $\partial=0$ and $B$ has nonzero cohomology.\\

In \cite{OR} the authors remark that the fact that in the Hermitian symmetric case $(\HH^*(\gr E),c_*)$ is a complex should be in principle a consequence
of the relations of the quiver. Our conjecture, motivated by examples similar to the ones we just described, is that the result of \cite{OR} 
does not generalize to the non-Hermitian symmetric case.

\subsection*{Acknowledgements.} \emph{This paper is part of my PhD thesis and it would not have been possible without my advisor Giorgio Ottaviani. 
I am deeply indebted to him for all he has taught me, and for the time and energies he has spent for me. 
I would also like to thank my mentor JM Landsberg for his help, and Jerzy Weyman, Jarek Buczynski and Andreas Cap for several interesting discussions.}

\bibliographystyle{amsalpha}
\bibliography{sections}

\end{document}